\documentclass[letterpaper, 10pt, conference]{ieeeconf}  
\IEEEoverridecommandlockouts                              
\usepackage[T1]{fontenc}

\usepackage[applemac]{inputenc}
\usepackage{graphics} 
\usepackage{epsfig} 
\usepackage{epstopdf}
\usepackage{amsmath} 
\usepackage{amssymb}  
\usepackage{epstopdf}
\usepackage{graphicx,xcolor}
\usepackage{caption}
\usepackage{subcaption}
\usepackage{courier}
\usepackage{mathrsfs}
\usepackage{pgfplots}
\DeclareMathOperator*{\argmin}{arg\,min}
\usepackage{tikz}
\usetikzlibrary{shapes,arrows}

\tikzstyle{block} = [draw, fill=blue!10, rectangle, 
    minimum height=4em, minimum width=7em]
\tikzstyle{sum} = [draw, fill=blue!10, circle, node distance=1cm]
\tikzstyle{input} = [coordinate]
\tikzstyle{output} = [coordinate]
\tikzstyle{pinstyle} = [pin edge={to-,thin,black}]

\newtheorem{thm}{Theorem}

\newtheorem{assum}{Assumption}
\newtheorem{prop}{Proposition}

\title{\LARGE \bf
Distributed Synthesis Using Accelerated ADMM
}

\author{Mohamadreza Ahmadi, Murat Cubuktepe, Ufuk Topcu, and Takashi Tanaka
\thanks{M. Ahmadi, M. Cubuktepe, U. Topcu and T. Tanaka are with the Department of Aerospace Engineering and Engineering Mechanics, and the Institute for Computational Engineering and Sciences (ICES), University of Texas at Austin, 201 E 24th St, Austin, TX 78712. e-mail: (\{mrahmadi, mcubuktepe, utopcu, ttanaka\}@utexas.edu). This work was supported by projects AFRL UTC 17-S8401-10-C1 and ONR N000141613165.
}
}

\begin{document}

\maketitle
\thispagestyle{empty}
\pagestyle{empty}

\begin{abstract}

We propose a convex distributed optimization algorithm for synthesizing robust controllers for large-scale continuous time systems subject to exogenous disturbances. Given a large scale system, instead of solving the larger centralized synthesis task, we decompose the problem into a set of smaller synthesis problems for the local subsystems with a given interconnection topology. Hence, the synthesis problem is constrained to the sparsity pattern dictated by the interconnection topology. To this end, for each subsystem, we solve a local dissipation inequality and then check a small-gain like condition for the overall system.  To minimize the effect of disturbances, we consider the $\mathrm{H}_\infty$ synthesis problems. We instantiate the distributed synthesis method using  accelerated alternating direction method of multipliers (ADMM) with convergence rate $O(\frac{1}{k^2})$ with $k$ being the number of iterations.  \end{abstract}

\section{INTRODUCTION}

Large scale systems are ubiquitous in nature and engineering.  Examples of such systems run the gamut of biochemical networks~\cite{PRESCOTT2014113} to smart grids~\cite{6740090}. The synthesis of control laws for large scale systems, however, is fettered by several challenges, despite the availability of well-known solutions to the LQR, $\mathrm{H}_2$, and $\mathrm{H}_\infty$ synthesis problem for linear systems~\cite{dullerud2005course}. 

Conventional feedback synthesis techniques are implemented in a centralized manner and often lead to controller gains {that require dense interconnection topology}. {Such} controllers require full communication between the subsystems that may not be practical. Therefore, significant research has been carried out in the literature to  synthesize controllers with a given sparsity pattern. In general, the linear synthesis problem subject to a given sparsity pattern in the controller is NP-hard~\cite{S0363012994272630}. However, explicit
solutions for special structures~\cite{1556730,6980067,shah2013cal}, convex relaxations~\cite{6948355,7464306} and sub-optimal non-convex methods~\cite{6497509,5893917} were proposed in the literature. 

Furthermore, the computational cost of synthesizing centralized controllers become prohibitive for large scale systems. In the analysis domain, several attempts have been made to decompose the centralized analysis problem using dissipativity theory based on the notion that if the subsystems of an interconnected large scale system satisfy some dissipativity property, the overall interconnected system satisfies some stability or input-output property~\cite{willems72}. At the expense of some degree of conservatism, \cite{6070953,7583731} propose  methods for analyzing large-scale dynamical systems by decomposing them into coupled {smaller} subsystems that are significantly simpler from a computational perspective. Furthermore, distributed optimization techniques can be used to carry out the computations in parallel.   In~\cite{5400309}, it was demonstrated that one can study the stability of a large scale system by  checking a number of dissipation inequalities for smaller local subsystems and then verifying a global gain condition, where the dual decomposition technique was used. Reference~\cite{Meissen201555} gives a decompositional dissipativity analysis methodology for a large scale system based on ADMM. \cite{wang2016system,wang2017separable} consider distributed synthesis for finite horizon discrete time linear systems, and they employ an ADMM based approach for a class of problems that can be decomposed into subproblems.



Distributed optimization methods such as ADMM~\cite{boyd2011distributed} are used to solve large-scale convex optimization problems by decomposing them into a set of smaller problems. These methods are most useful when interior-point methods do not scale well due to solving  a large linear system on the order of number of variables. However, these algorithms can be very slow to converge to high accuracy, unlike interior-point methods, where high accuracy can be attained in a reasonable amount of iterations. To attain faster convergence to high accuracy,  \emph{accelerated} versions of first order methods have been proposed in the literature \cite{hale2008fixed, becker2011nesta, nesterov2013gradient, beck2009fast, chen2012fast}. These methods achieve $O(\frac{1}{k^2})$ convergence rate after $k$ iterations, which is shown to be indeed optimal for a first order method \cite{nesterov1983method}. The main drawback of these approaches is, they usually require the objective function to be differentiable, which disallows constraints in the optimization problem. 

In this paper, we consider large-scale systems with a given interconnection topology. The interconnection topology of the system dictates a sparsity pattern to the global controller. Given this controller structure, we tackle the global synthesis problem by synthesizing local controllers such that some global system property, namely, stability or $\mathrm{H}_\infty$ performance is guaranteed. To this end, we use a set of local dissipation inequalities and adopt a block-diagonal global Lyapunov function structure.  This decomposes the global synthesis problem into a number of smaller ones for local subsystems. Moreover, for large scale problems, we provide a computational formulation based on accelerated ADMM. We use a smoothing approach~\cite{nesterov2005smooth,becker2011templates} to achieve a faster convergence rate than the conventional ADMM. Specifically, when the objective function is strongly convex, the ADMM algorithm can be modified with an extra acceleration step to attain $O(\frac{1}{k^2})$ convergence rate~\cite{goldstein2014fast}. We show the applicability of our approach in numerical experiments.

The rest of the paper is organized as follows. The next section presents  the problem formulation and some preliminary results. In Section~\ref{sec:main}, we describe the distributed stabilizing  and $\mathrm{H}_\infty$ synthesis method based on dissipativity theory. In Section~\ref{sec:admm}, we bring forward the computational formulation based on accelerated ADMM. Two examples are given in Section~\ref{sec:examples} to illustrate the proposed method. Finally, Section~\ref{sec:conclusions} concludes the paper and provides directions for future research.


\textbf{Notation:}
The notations employed in this paper are relatively straightforward. $\mathbb{R}_{\ge 0}$ denotes the set $[0,\infty)$. $\| \cdot \|$ denotes the Euclidean vector norm on $Q$ and $\langle \cdot \rangle$ the inner product. We denote the set of $n \times n$ real symmetric matrices as $\mathbb{S}^n$. For a matrix $A \in \mathbb{R}^{m \times n}$, $A^\dagger$ denotes the pseudo-inverse of $A$. Note that the pseudo-inverse of $A$ always exists and is unique~\cite{Golub1996}. For a function $f:A\to B$, $f \in L^p(A, B)$, $1\le p  < \infty$, implies that $\left( \int_A |f(t)|^p dt \right)^{\frac{1}{p}} < \infty$ and $\sup_{t \in A} |f(t)| < \infty$ for $p=\infty$

\section{Problem Formulation}

We consider the controlled linear dynamical systems described~by
\begin{align} \label{eq:subsystems}
\mathcal{G}: \begin{cases} \dot{x} =  Ax+Bu+Gw,  \\  
y =  Cx,
\end{cases}
\end{align}
where, $x \in \mathcal{X} \subseteq \mathbb{R}^n$ are the states of the system, $y \in \mathcal{Y} \subseteq \mathbb{R}^{n_y}$ are the outputs of the system, $w \in \mathcal{W} \subseteq \mathbb{R}^{n_w}$ are the exogenous disturbances and $u \in \mathcal{U} \subseteq \mathbb{R}^m$ are the control signals. The matrices $A: \mathbb{R}^{n \times n} $, $B \in \mathbb{R}^{n\times m}$, $G \in \mathbb{R}^{n\times n_w}$ and $C \in \mathbb{R}^{n_y \times n}$.

\begin{figure}
\centering
\includegraphics[width=6cm]{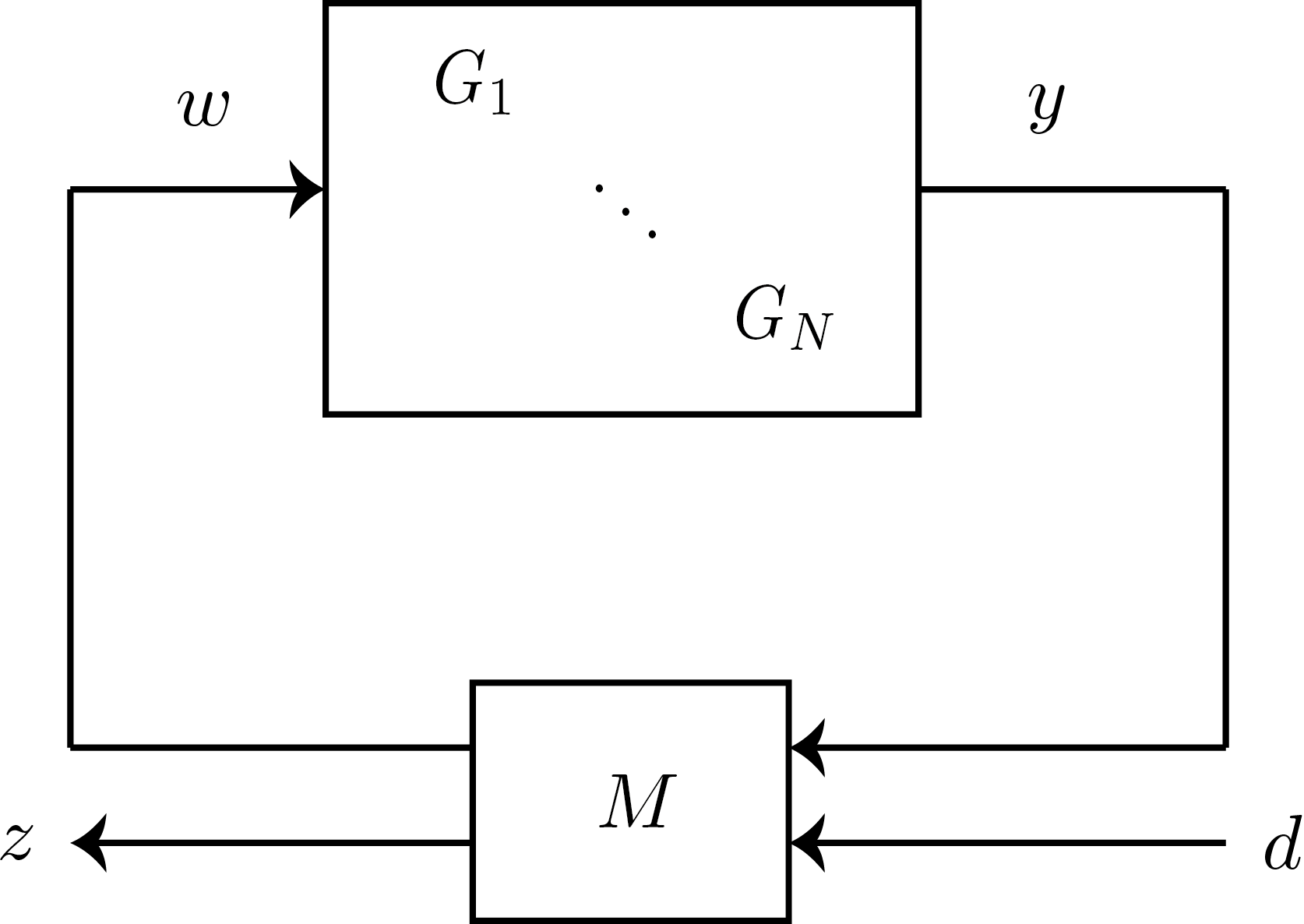}
\caption{Interconnected system with input $d$ and output $z$.}\label{fig1}
\end{figure}

We consider interconnected systems as illustrated in Fig.~\ref{fig1}, where the subsystems $\{ \mathcal{G}_i \}_{i=1}^N$ are known and have dynamics in the form of~\eqref{eq:subsystems}. We associate each subsystem with a set of matrices $\{ A_i,B_i,G_i,C_i\}$ and $x_i \in \mathcal{X}_i \subseteq \mathbb{R}^{n_i}$, $w_i \in \mathcal{W}_i \subseteq \mathbb{R}^{n_w^i}$ and $y_i \in \mathcal{Y}_i \subseteq \mathbb{R}^{n_y^i}$. The static interconnection is characterized by a matrix $M \in \mathbb{R}^{n_w} \times \mathbb{R}^{n_y}$ where $n = \sum_{i=1}^N n_i$, $n_w = \sum_{i=1}^N n_w^i$ and $n_y = \sum_{i=1}^N n_y^i$. That is, $M$ satisfies 
\begin{equation} \label{eq:interconnection}
\begin{bmatrix} w \\ z \end{bmatrix} = M \begin{bmatrix} y \\ d \end{bmatrix},
\end{equation}
where $d \in \mathbb{R}^{n_d}$ and $z \in \mathbb{R}^{n_z}$. We assume this interconnection is well-posed, i.e., for all $d \in L_{2e}$ and initial condition $x(0) \in \mathbb{R}^n$, there exist unique $z,w,y \in L_{2e}$ that casually depend on $d$. Furthermore, we define
$$
M=
\left[
\begin{array}{c|c}
M_{wy}& M_{wd} \\
\hline
M_{zy} & M_{zd}
\end{array}
\right],
$$
where $M_{wy} \in \mathbb{R}^{n_w \times n_y}$, $M_{wd} \in \mathbb{R}^{n_w \times n_d}$, $M_{zy} \in \mathbb{R}^{n_z \times n_y}$, and $M_{zd} \in \mathbb{R}^{n_z \times n_d}$

The local  and global supply rates, $W_i(w_i,y_i)$ and $W(d,z)$, respectively, are defined by quadratic functions. That is, 
\begin{equation} \label{fsfFF}
W(x,d,z) = \begin{bmatrix}  d \\ z \end{bmatrix}^\prime S \begin{bmatrix}  d \\ z \end{bmatrix}.
\end{equation}

In~\cite{Meissen201555}, the authors, inspired by the work~\cite{5400309}, showed that certifying the dissipativity of an overall interconnected system can be concluded if each of the subsystems satisfy the local dissipativity property. Let 
\begin{equation}
\mathcal{L}_i = \bigg \{ S_i \mid \text{$\begin{bmatrix} w_i \\ y_i \end{bmatrix}^\prime S_i \begin{bmatrix} w_i \\ y_i \end{bmatrix}$} \le 0 \bigg\},
\end{equation}
and 
\begin{equation} \label{sddccxcxglobal}
\mathcal{L} = \bigg\{ \{S_i\}_{i=1}^N \mid \begin{bmatrix} M \\ I_{n_y}  \end{bmatrix}^\prime P_\pi^\prime Q P_\pi \begin{bmatrix} M \\ I_{n_y}   \end{bmatrix} < 0 \bigg\},
\end{equation}
where $Q = diag(S_1,\ldots,S_N,-S)$ and $P_\pi$ is a permutation matrix defined by
\begin{equation}
\begin{bmatrix} w_1 \\ y_1 \\ \vdots \\ w_N \\ y_N \\d \\z \end{bmatrix} = P_\pi \begin{bmatrix} w \\ z  \\y \\d \end{bmatrix} .
\end{equation}

\begin{prop}[Proposition 1, \cite{Meissen201555}] \label{prop1}
Consider the interconnection of $N$ subsystems as given in~\eqref{eq:interconnection} with the global supply rate~\eqref{fsfFF}. If there exists $\{ S_i \}_{i=1}^N$ satisfying
\begin{equation}\label{ddsdsdsdsdsccdcdxx}
S_i \in \mathcal{L}_i,~~i=1,\ldots,N,
\end{equation}
and
\begin{equation} \label{eq:global}
(S_1,\ldots, S_N, -S) \in \mathcal{L},
\end{equation}
then the interconnected system is dissipative with respect to the global supply rate. A storage function certifying  global dissipativity is $V(x) = \sum_{i=1}^N V_i(x_i)$, where $V_i$ is the storage function certifying dissipativity of subsystem $i$ as in $\mathcal{L}_i$.
\end{prop}

\section{Distributed Synthesis Using Dissipativity}\label{sec:main}

The dynamics of each subsystem of $\mathcal{G}$ is characterized by
\begin{align} \label{eq:subsystemslin}
\mathcal{G}_i: \begin{cases}
\dot{x}_i =  A_ix_i+B_iu_i+G_iw_i,  \\  
y_i =  C_ix_i,
\end{cases}
\end{align}
with $i=1,2,\ldots,N$. Using~\eqref{eq:interconnection}, the overall system can be described as
\begin{align}
\mathcal{G}: \begin{cases}
\dot{x} = (A+GM_{wy}C) ~x + Bu+M_{wd}~d, \\
z = M_{zy}C~x + M_{zd} ~d,
\end{cases}
\end{align}
where $A= diag(A_1,\ldots,A_N) \in \mathbb{R}^{n \times n}$, $B= diag(B_1,\ldots,B_N) \in \mathbb{R}^{n \times n_u}$, $G= diag(G_1,\ldots,G_N) \in \mathbb{R}^{n \times n_w}$ and $C = diag(C_1,\ldots,C_N)\in \mathbb{R}^{n_y \times n}$. Although $A$ has a block-diagonal structure, the overall system matrix $A+GM_{wy}C$ does not necessarily has any special sparsity pattern or structure. Moreover, the controller takes the form of 
$$
u = Kx,
$$
where $K = diag(K_1,\ldots,K_N) \in \mathbb{R}^{n_u \times n}$. Therefore, $K$ has a sparsity pattern inherited from the topology of the interconnections. In other words, for any subsystem $\mathcal{G}_i$, the local controllers have only access  to the states $x_i$ of the local subsystem. 

In the following, we formulate conditions based on linear matrix inequalities to synthesize stabilizing and robust feedback controllers $u_i=K_ix_i$ for the subsystem~\eqref{eq:subsystemslin} such that the overall system satisfies some dissipative property. The performance of these local controllers are certified by local storage functions $V(x_i) = x_i^\prime P_i x_i$, $i=1,2,\ldots,N$ based on Proposition~\ref{prop1}.

%

\subsection{Distributed  Synthesis of Stabilizing Controllers} \label{sec:stable}

The following result states that the distributed search for a set of stabilizing controllers is a convex problem. Before we state the theorem, we define 
\begin{equation}
S_i=
\left[
\begin{array}{c|c}
S_i^{11} & S_i^{12} \\
\hline
S_i^{21} & S_i^{22}
\end{array}
\right],
\end{equation}
where $S_i^{11} \in \mathbb{S}^{n_w^i}$, $S_i^{12} \in \mathbb{R}^{n_w^i \times n_y^i}$, $S_i^{11} \in \mathbb{R}^{n_y^i \times n_w^i}$ and $S_i^{22} \in \mathbb{S}^{n_y^i}$.

\begin{thm} \label{thmmain}
Consider subsystems~\eqref{eq:subsystemslin} and the interconnection topology given by~\eqref{eq:interconnection}, with $d,z\equiv0$. If there exist  families of matrices $\{S_i\}_{i=1}^N$, $\{P_i\}_{i=1}^N$ ($P_i \in \mathbb{S}^{n_i}$) and $\{Y_i\}_{i=1}^N$ ($Y_i \in \mathbb{R}^{n_i \times n_i}$) such that~\eqref{eq:global} is satisfied with $S=0$, and
\begin{equation} \label{con1}
P_i > 0,\quad i=1,2,\ldots, N,
\end{equation}
\begin{multline} \label{con2}
\left[ \begin{matrix} A_i^\prime P_i + P_i A_i + Y_i^\prime + Y_i - C_i^\prime S_i^{22} C_i \\ G_i^\prime P_i -S_i^{12} C_i
\end{matrix} \right. \\
\left. \begin{matrix}
P_iG_i -C_iS_i^{21} \\
-S_i^{11}
\end{matrix}  \right] \le 0,
\end{multline}
for $i=1,2,\ldots, N$, then the local controllers
\begin{equation} \label{eq:controllers}
K_i = B_i^\dagger P_i^{-1} Y_i,
\end{equation}
 render the overall system asymptotically stable.
\end{thm}
\begin{proof}
Let $V_i(x) =x_i^\prime P_i x_i$, $i=1,2,\ldots,N$ be a family of local candidate storage functions. Inequalities \eqref{con1} ensures that $V_i$'s are positive definite functions of $x_i$. Computing the time derivative of $V_i$ gives
\begin{equation*}
\frac{d V_i}{dt} = \dot{x}_i^\prime P_i x_i + {x}_i^\prime P_i \dot{x}_i.
\end{equation*}
Substituting \eqref{eq:subsystemslin} and $u_i = K_i x_i$ yields
\begin{multline*}
\frac{d V_i}{dt} = \left(A_ix_i +B_i K_i x_i + G_i w_i \right)^\prime P_i x_i \\
+ {x}_i^\prime P_i  \left(A_ix_i +B_i K_i x_i + G_i w_i \right)  \\
= x_i^\prime \left(A_i^\prime P_i + P_i A_i +K_i^\prime B_i^\prime P_i + P_iB_i K_i \right) x_i \\+ w_i^\prime G_i^\prime P_i x_i + x_i^\prime P_i G_i w_i.
\end{multline*}
Writing the last line of the above expression in quadratic form and substituting \eqref{eq:controllers} gives
\begin{multline} \label{ssfffdd}
\frac{d V_i}{dt} = 
\begin{bmatrix} x_i \\ w_i \end{bmatrix}^\prime \begin{bmatrix} A_i^\prime P_i + P_i A_i +Y^\prime_i + Y_i & P_iG_i \\ G_i^\prime P_i & 0  \end{bmatrix} \begin{bmatrix} x_i \\ w_i \end{bmatrix}
\end{multline}
Since~\eqref{con2} holds, if we multiply it left and right by $\begin{bmatrix} x_i \\ w_i \end{bmatrix}^\prime$ and $\begin{bmatrix} x_i \\ w_i \end{bmatrix}$, respectively, we obtain
\begin{equation}
\frac{d V_i}{dt} \le \begin{bmatrix} w_i \\ C_i x \end{bmatrix}^\prime \left[
\begin{array}{c|c}
S_i^{11} & S_i^{12} \\
\hline
S_i^{21} & S_i^{22}
\end{array}
\right] \begin{bmatrix} w_i \\ C_i x \end{bmatrix},
\end{equation}
in which we used~\eqref{ssfffdd}. Summing up the terms in the above expression for all $i$ gives
\begin{equation} \label{dscccs}
\frac{d V}{dt} \le  \sum_{i=1}^N \begin{bmatrix} w_i \\ y_i\end{bmatrix}^\prime 
S_i
\begin{bmatrix} w_i \\ y_i \end{bmatrix}.
\end{equation}
Furthermore, since~\eqref{eq:global} holds with $S=0$, multiplying the LMI in~\eqref{sddccxcxglobal} from left and right by $\begin{bmatrix} y \\ d\end{bmatrix}^\prime$ and $\begin{bmatrix} y \\ d\end{bmatrix}$, respectively, and using the interconnection topology relation~\eqref{eq:interconnection}, we have
$$
  \sum_{i=1}^N \begin{bmatrix} w_i \\ y_i\end{bmatrix}^\prime 
S_i
\begin{bmatrix} w_i \\ y_i \end{bmatrix} < 0.
$$
Therefore,
$
\frac{d V}{dt} \le  \sum_{i=1}^N \begin{bmatrix} w_i \\ y_i\end{bmatrix}^\prime 
S_i
\begin{bmatrix} w_i \\ y_i \end{bmatrix} < 0.
$
\end{proof}

\subsection{Distributed $\mathrm{H}_\infty$-Synthesis} \label{sec:robust}

Theorem~\ref{thmmain} brings forward conditions under which the interconnected system can be stabilized by synthesizing a set of local controllers $\{K_i \}_{i=1}^N$. Depending on the choice of matrix $Q$, in particular $S$, in the global dissipativity constraint~\eqref{sddccxcxglobal}, we can have different synthesis performances. The next proposition states that we can synthesize local controllers such that the global $\mathrm{H}_\infty$-norm from the inputs $d$ to $z$ is minimized. 

\begin{assum}[Zero-State Detectability]
Let $n_y < n$. It holds that
$$
rank \left( M_{zy} \begin{bmatrix} C_1 & 0& 0 \\ 0& \ddots & 0\\ 0& 0& C_N \end{bmatrix} \right) = n_y.
$$
\end{assum}

Note that the above assumption means that $y \equiv 0$ implies $x \equiv 0$. This property  will be used in the Proposition below to prove asymptotic stability. 

\begin{prop}
Consider subsystems~\eqref{eq:subsystemslin} and the interconnection topology given by~\eqref{eq:interconnection}. Let Assumption~1 holds and $d \in L_{2}$. If there exists a positive constant $\eta$, and families of matrices $\{S_i\}_{i=1}^N$,  $\{P_i\}_{i=1}^N$ ($P_i \in \mathbb{S}^{n_i}$) and $\{Y_i\}_{i=1}^N$ ($Y_i \in \mathbb{R}^{n_i \times n_i}$) that solves the following optimization problem 
\begin{eqnarray}
&\displaystyle \min_{\{S_i\}_{i=1}^N, \{P_i\}_{i=1}^N, \{Y_i\}_{i=1}^N} \eta & \nonumber \\
&\text{subject to \eqref{con1}, \eqref{con2}, and \eqref{eq:global}}& \label{eq:globalobj}
\end{eqnarray}
with 
\begin{equation}\label{rreweewewew}
S = \begin{bmatrix} \eta I_{n_d} & 0 \\ 0 & -I_{n_z} \end{bmatrix},
\end{equation}
then, the local controllers \eqref{eq:controllers} render the interconnected system asymptotically stable for $d \equiv 0$. Furthermore, when $x(0)=0$, we have
\begin{equation} \label{hinfnorm}
\| G \|_{\mathrm{H}_\infty}:= \inf_{\|d\|_{L_2} \neq 0} \frac{\|z\|_{L_2}}{\|d\|_{L_2}} = \sqrt{\eta}.
\end{equation}
\end{prop}
\vspace{.8cm}
\begin{proof}
The proof follows the same lines as the proof of Theorem~\ref{thmmain}. Since~\eqref{eq:global} holds with $S$ given by \eqref{rreweewewew}, multiplying the LMI in~\eqref{sddccxcxglobal} from left and right by $\begin{bmatrix} y \\ d\end{bmatrix}^\prime$ and $\begin{bmatrix} y \\ d\end{bmatrix}$, respectively, and using the interconnection topology relation~\eqref{eq:interconnection}, we obtain
$$
  \sum_{i=1}^N \begin{bmatrix} w_i \\ y_i\end{bmatrix}^\prime 
S_i
\begin{bmatrix} w_i \\ y_i \end{bmatrix} < \begin{bmatrix} d \\ z\end{bmatrix}^\prime 
\begin{bmatrix} \eta I_{n_d} & 0 \\ 0 & -I_{n_z} \end{bmatrix}
\begin{bmatrix} d \\ z \end{bmatrix}.
$$
From~\eqref{con2} and then~\eqref{dscccs}, we have
\begin{equation} \label{dccscscdssd}
\frac{dV}{dt} < \eta |d|^2 - |z|^2.
\end{equation}
If  $d\equiv 0$, we obtain
$$
\frac{dV}{dt} <  - |z|^2.
$$
From Assumption~1, we infer that $|z|^2=0$ if and only if $x \equiv 0$. Thus, from LaSalle's Invariance Theorem~\cite{khalil1996noninear}, we deduce that the system is asymptotically stable.

Moreover, integrating both sides of inequality~\eqref{dccscscdssd} from time $0$ to $\infty$ gives
$$
V(x(\infty)) - V(x(0)) < \eta \int_0^\infty |d(t)|^2~dt - \int_0^\infty |z(t)|^2~dt. 
$$
From the fact that the storage function is positive definite (it is the finite sum of positive definite functions), we infer that, for $x(0)\equiv 0$, $V(x(0)) = 0$ and we have
$$
0< \eta \int_0^\infty |d(t)|^2~dt - \int_0^\infty |z(t)|^2~dt.
$$
That is,
$$
\|z\|^2_{L_2} < \eta \|d\|^2_{L_2}.
$$
Hence, minimization over $\eta$ gives~\eqref{hinfnorm}.
\end{proof}

The formulation presented here can also be extended to accommodate class of uncertainties and nonlinearities  captured by Integral Quadratic Constraints (see Appendix).

\subsection{Discussion on Conservativeness}\label{sec:conserv}

In the following, we briefly discuss  the conservativeness of the proposed method in Section~\ref{sec:stable} and Section~\ref{sec:robust}. 

Given the subsystems $\mathcal{G}_1,\ldots, \mathcal{G}_N$, the global Lyapunov function we consider is given by
$$
V(x) = \sum_{i=1}^N V_i(x_i) =\begin{bmatrix} x_1 \\ \vdots \\ x_N  \end{bmatrix}^\prime
\begin{bmatrix} P_1 & 0 & 0 \\ 0 & \ddots & 0 \\ 0 & 0& P_N  \end{bmatrix} \begin{bmatrix} x_1 \\ \vdots \\ x_N  \end{bmatrix}.
$$
This is indeed a conservative structure, since the cross terms in the quadratic Lyapunov function are zero. In fact, Theorem~\ref{thmmain} searches over all \textit{block-diagonal Lyapunov functions} in a distributed manner using dissipativity theory. Block-diagonal Lyapunov functions were also considered in~\cite{7979611} to design local controllers based on chordal decomposition techniques for block-chordal
sparse matrices. In addition, it was demonstrated in~\cite{yang17}, that block-diagonal Lyapunov functions lead to sparsity
invariance and can be used design structured controllers.

Although converse theorems for diagonal Lyapunov functions  have received tremendous attention in the literature (especially for positive systems~\cite{RANTZER201572,VARGA19761}), studies on finding the class of stable systems that admit block-diagonal Lyapunov functions are relatively few and far between~\cite{feingold1962}.   Recently, in~\cite{aivar17}, the authors  define a suitable comparison matrix, and then demonstrate that if the comparison matrix is scaled diagonal dominant,  the stability of a block system is equivalent to the existence of a block diagonal Lyapunov function.

\section{Computational Formulation Using Accelerated ADMM}\label{sec:admm}
For small-scale systems, we can solve the feasibility and optimization problems in Theorem 1 and Proposition 2 using publicly available SDP solvers like MOSEK \cite{andersen2012mosek}, SeDuMi \cite{sturm1999using} or SDPT3 \cite{toh1999sdpt3}. But, these SDP solvers do not scale well for larger problems. For larger problems, the ADMM algorithm \cite{boyd2011distributed} allows us to decompose convex optimization problems into a set of smaller problems. A generic convex optimization problem
\begin{align}
 \text{minimize} &\quad  f(y)\nonumber\\
 \text{subject to} &\quad  y\in \mathit{C},
\end{align}
where $x \in \mathbb{R}^n$, $f$ is a convex function, and $C$ is a convex set, can be written in ADMM form as 

\begin{align}
 \text{minimize} &\quad  f(y)+g(v)\nonumber \\ 
 \text{subject to} &\quad  y=v, \label{eq:admm2}
\end{align}
where $g$ is the indicator function of $\mathit{C}$.

Using the above form,  problem \eqref{eq:global} can be written in ADMM form with $f(y)$ is defined as sum of $\mu$ and the indicator function of \eqref{con1} and \eqref{con2}, and $g(v)$ is defined as the indicator function of \eqref{eq:global}. Then, the scaled form of ADMM form for problem in \eqref{eq:admm2} is

\begin{align*}
& y^{k+1} = \argmin_{y_i\in \mathcal{L}_i} f(y) + (\rho /2)\vert\vert y-v^k+u^k\vert\vert^2_2,\\
& v^{k+1}=\argmin_{v \in \mathcal{L}} g(v) + (\rho /2)\vert\vert y^{k+1}-v+u^k\vert\vert^2_2,\\
& u^{k+1}=u^k+y^{k+1}-v^{k+1},
\end{align*}
where $y$ and $v$ are the vectorized form of the matrices $\{S_i\}_{i=1}^N, \{P_i\}_{i=1}^N, \{Y_i\}_{i=1}^N$, $u$ is the scaled dual variable and $\rho > 0$ is the \emph{penalty parameter}. Since $f(y)$ is separable for each subsystem, the ADMM algorithm can be parallelized as follows:

\begin{align*}
& y_i^{k+1} = \argmin_{y_i\in \mathcal{L}_i} f_i(y) + (\rho /2)\vert\vert y_i-v_i^k+u_i^k\vert\vert^2_2,\\
& v^{k+1}=\argmin_{v \in \mathcal{L}} g(v) + (\rho /2)\vert\vert y^{k+1}-v+u^k\vert\vert^2_2,\\
& u^{k+1}=u^k+y^{k+1}-v^{k+1},
\end{align*}

Under mild assumptions, the ADMM algorithm converges \cite{boyd2011distributed}, but the convergence is only asymptotic, therefore it may require many iterations to achieve sufficient accuracy.

\subsection{Accelerated ADMM}

For faster convergence, the so-called \textit{accelerated} versions of similar first order algorithms have been proposed in the literature \cite{hale2008fixed, becker2011nesta, nesterov2013gradient, beck2009fast, chen2012fast}, and the methods achieve $O(\frac{1}{k^2})$ convergence after $k$ iterations, which is shown to be optimal for a first order method \cite{nesterov1983method}. The main drawback of these approaches is, they usually require the function $f(y)$ to be differentiable with a known Lipschitz constant on the $\nabla f(y)$, which does not exist when the problem has constraints. For the case when $f(y)$ or $g(y)$ is not strongly convex or smooth, smoothing approaches have been used \cite{nesterov2005smooth,becker2011templates} to improve convergence. However, to the best of our knowledge, these methods have not been applied in distributed controller synthesis problems.

Consider the following perturbation of \eqref{eq:globalobj}:

\begin{eqnarray}
&\displaystyle \min_{\{S_i\}_{i=1}^N, \{P_i\}_{i=1}^N, \{Y_i\}_{i=1}^N} \eta +\mu ~d_i(S_i,P_i,Y_i)& \nonumber \\
&\text{subject to \eqref{con1}, \eqref{con2}, and \eqref{eq:global}}& \label{eq:globalobj_approx}
\end{eqnarray}
for some fixed smoothing parameter $\mu >0$ and a strongly convex function $d$ that satisfies

\begin{equation}
d(y)\geq d(y_0) + \dfrac{1}{2}\vert\vert y-y_0\vert\vert^2_2
\end{equation}
for some point $y_0 \in \mathbb{R}^{n}.$ Specifically, we choose $d_i=\lVert S_i\rVert_F+\lVert P_i\rVert_F+\lVert Y_i\rVert_F$, where  $\| \cdot \|_F$ is the Frobenius norm.  For some problems, it is shown that for small enough $\mu$, the approximate problem \eqref{eq:globalobj_approx} is equivalent to the original problem \eqref{eq:globalobj} \cite{becker2011templates}. 

When $f(x)$ and $g(x)$ are strongly convex, the ADMM algorithm can be modified with an acceleration step to achieve $O(\frac{1}{k^2})$ convergence after $k$ iterations  \cite{goldstein2014fast}. Then, the accelerated ADMM algorithm is

\begin{align*}
& y_i^{k} = \argmin_{y_i\in \mathcal{L}_i} f_i(y) + (\rho /2)\vert\vert y_i-\bar{v}_i^k+\bar{u}_i^k\vert\vert^2_2,\\
& v^{k}=\argmin_{v \in \mathcal{L}} g(z) + (\rho /2)\vert\vert y^{k}-v+\bar{u}^k\vert\vert^2_2,\\
& u^{k}=\bar{u}^k+y^{k}-v^{k},\\
& \alpha_{k+1}= \dfrac{1+\sqrt{1+4\alpha^2_k}}{2}\\
& \bar{v}^{k+1}=v_k+\dfrac{\alpha_k -1}{\alpha_{k+1}}(v^k-v^{k-1})\\
& \bar{u}^{k+1}=u^{k}+\dfrac{\alpha_k -1}{\alpha_{k+1}}(u^k-u^{k-1}),
\end{align*}
where $\rho$ is a positive constant that satisfies $\rho\leq \mu$, and~$\alpha_1=1$.

Note that $y$ update can be carried out in parallel, while achieving $O(\frac{1}{k^2})$ convergence, which cannot be achieved by the standard ADMM or accelerated proximal methods due to constraints in the problem.

\section{Numerical Experiments}\label{sec:examples}

In this section, we illustrate the proposed distributed synthesis method using two examples. 
The  example pertains to a system with 100-states, where we compare the convergence rate of ADMM with accelerated ADMM. We implemented both standard ADMM and accelerated ADMM algorithms in MATLAB using the CVX toolbox \cite{grant2008cvx} and MOSEK \cite{andersen2012mosek} to solve SDP problems.
\begin{figure}
\centering
\includegraphics[width=5cm]{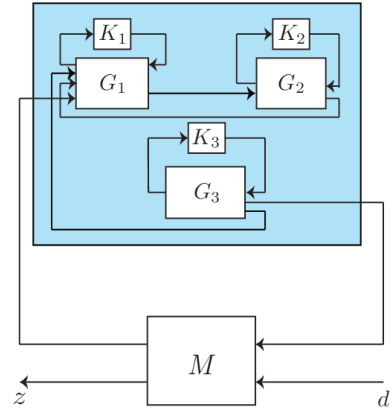}
\caption{The interconnected system in Example I.}\label{fig22}
\end{figure}

\subsection{Example I}

We consider a modified version of the example in \cite{topcu2009compositional} as illustrated in Fig~\ref{fig22}. For $i = 1, 2, 3$, the subsystems $\mathcal{G}_i$ are characterized as follows:
$$
\mathcal{G}_1: \begin{cases}
\dot{x}_1 =\begin{bmatrix}
4 & 0\\
2 & -2
\end{bmatrix}x_1+I_2\begin{bmatrix}
w_2\\
w_4+d
\end{bmatrix}+I_2 u_1,\\
w_1 =0.5\begin{bmatrix}
1 & 1
\end{bmatrix}x_1,\\
\end{cases}
$$
$$
\mathcal{G}_2: \begin{cases}
\dot{x}_2 =\begin{bmatrix}
8 & 0\\
12 & -2
\end{bmatrix}x_2+\begin{bmatrix}
1\\
1
\end{bmatrix}w_1+I_2u_2,\\
\begin{bmatrix}w_2\\w_3\end{bmatrix} =0.5I_2x_2,\\
\end{cases}
$$
$$
\mathcal{G}_3: \begin{cases} \dot{x}_3 =\begin{bmatrix}
2 & 0\\
2 & -2
\end{bmatrix}x_3+\begin{bmatrix}
1\\
1
\end{bmatrix}w_3+I_2 u_3,\\
w_3 =0.4\begin{bmatrix}
1 & 1
\end{bmatrix}x_3.
\end{cases}
$$
The global input and output are given as

\begin{align*}
d=\begin{bmatrix}
0 & 1
\end{bmatrix}x_3 \quad \text{and} \quad z=w_4=0.4\begin{bmatrix}
1 & 1
\end{bmatrix}x_3.
\end{align*}
The overall system with zero-input has 3 eigenvalues with positive real parts, therefore the zero-input system is unstable. We apply the compositional approach underlined by the problem in \eqref{eq:globalobj_approx} to synthesize a controller to minimize the $\mathrm{H}_\infty$-norm between the output $z$ and input of $d$ subsystem 1. After 19 iterations with the accelerated ADMM method, the value of the objective is $\eta=4.9\cdot 10^{-4}$ with the following controllers:

\begin{align*}
K_1&=\begin{bmatrix}
-4.7760 &  -1.1078\\
   -1.1138 & -0.6947
\end{bmatrix},\\
 K_2&=\begin{bmatrix}
  -6.5288 & -4.5095\\
  -5.5751 & -3.8507
  \end{bmatrix},\\
   K_3&=\begin{bmatrix}
   -100.4125 &  -90.2510\\
   -56.2576 &  -50.5649
\end{bmatrix}.
\end{align*}

and following local Lyapunov functions:

\begin{align*}
P_1&=\begin{bmatrix}
  0.001 &  -2.7\cdot 10^{-6}\\
   -2.7\cdot 10^{-9} &   0.0010
\end{bmatrix},\\
 P_2&=\begin{bmatrix}
   0.0019 &  -9.87\cdot 10^{-4}\\
 -9.87\cdot 10^{-4} &  0.0021
  \end{bmatrix},\\
   P_3&=\begin{bmatrix}
   0.0045 &   -0.0064\\
    -0.0064 & 0.0129
\end{bmatrix}.
\end{align*}

\subsection{Example II}

In order to test and compare the different forms of ADMM methods discussed in Section~\ref{sec:admm}, we randomly generated $N=20$ unstable and controllable LTI subsystems, each with 5 states, 2 inputs, and 2 outputs. The dynamics of each subsystem is characterized by the equations \eqref{eq:subsystemslin}, with the maximum real part for an eigenvalue of $A_i$ is normalized to~1, and random interconnection matrix $M$ that was randomly generated with 5$\%$ its entries being nonzero. 

 The iterative methods were initialized using $S_i^0=P_i^0=Y_i^0=U_i^0=I$. For each method, we plot the norm of \emph{primal residual} in Figure~\ref{primal}, which is defined as $r^{k}=x^k-x^k$, and it is the residual for primal feasibility. Also, we show the  norm of the \emph{dual residual} $s^k=\rho(z^k-z^{k-1})$ in Figure~\ref{dual}, which can be viewed as a residual for the dual feasibility condition.

We remark that the accelerated ADMM method significantly improves the convergence rate for primal and dual residuals compared to standard ADMM. After 20 iterations, the accelerated method achieves convergence within the tolerance $\epsilon=10^{-6}$, which is not achievable by the standard ADMM method after 50 iterations, and it may require many iterations for the standard ADMM method to achieve high accuracy.

%
%
\definecolor{mycolor1}{rgb}{0.00000,0.44700,0.74100}%
\definecolor{mycolor2}{rgb}{0.85000,0.32500,0.09800}%
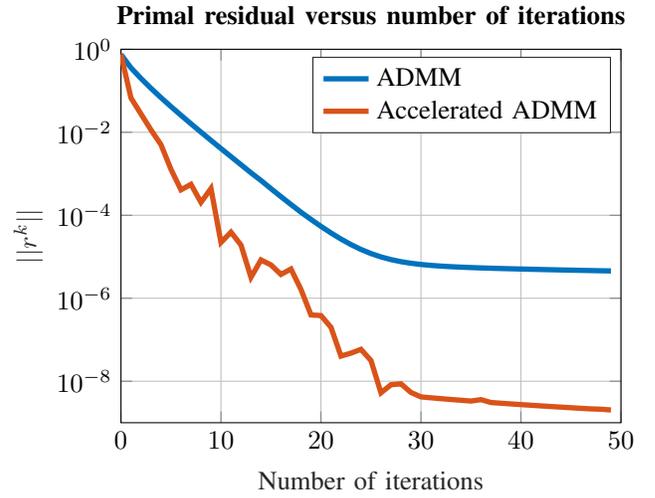
\begin{figure}[t!]

\begin{tikzpicture}

\begin{axis}[%
width=2.617in,
height=1.956in,
at={(0in,0in)},
scale only axis,
xmin=0,
xmax=50,
xlabel style={font=\color{white!15!black}},
xlabel={Number of iterations},
ymode=log,
ymin=1e-09,
ymax=1,
yminorticks=true,
xmajorgrids=true,
ymajorgrids=true,
ylabel style={font=\color{white!15!black}},
ylabel={$\vert \vert r^k \vert \vert$},
axis background/.style={fill=white},
title style={font=\bfseries},
title={Primal residual versus number of iterations},
legend style={legend cell align=left, align=left, draw=white!15!black}
]
\addplot [color=mycolor1, line width=2]
  table[row sep=crcr]{%
0	0.751991884610038\\
1	0.357781350551713\\
2	0.198997421533658\\
3	0.114891197302817\\
4	0.0685094093197512\\
5	0.041693374938756\\
6	0.0257327563144059\\
7	0.01604573030686\\
8	0.0100874170947716\\
9	0.00637950529737862\\
10	0.00405498318960514\\
11	0.00258753823941439\\
12	0.00165453904403427\\
13	0.00105556545032591\\
14	0.00068960217473674\\
15	0.000437878434462686\\
16	0.000280503388736723\\
17	0.000181867390474613\\
18	0.000118555713386077\\
19	7.94494645915716e-05\\
20	5.40549418109825e-05\\
21	3.75240053168519e-05\\
22	2.6732520240814e-05\\
23	1.96652885828018e-05\\
24	1.50199975857515e-05\\
25	1.19504007926634e-05\\
26	9.90856142053907e-06\\
27	8.54136119693403e-06\\
28	7.61086822840304e-06\\
29	6.96841267737509e-06\\
30	6.51706953038719e-06\\
31	6.19043826999837e-06\\
32	5.94833862592477e-06\\
33	5.76168320513378e-06\\
34	5.61281924775011e-06\\
35	5.48949606483194e-06\\
36	5.38494572919021e-06\\
37	5.29211219525047e-06\\
38	5.2084279251977e-06\\
39	5.13171716020708e-06\\
40	5.06028274373521e-06\\
41	4.99291122664376e-06\\
42	4.92985995758986e-06\\
43	4.86860181741916e-06\\
44	4.80979082083772e-06\\
45	4.75372708491706e-06\\
46	4.69924346398837e-06\\
47	4.6473376494384e-06\\
48	4.59636174003318e-06\\
49	4.54846658993368e-06\\
};
\addlegendentry{ADMM}

\addplot [color=mycolor2, line width=2]
  table[row sep=crcr]{%
0	0.751991884610038\\
1	0.0673032558341399\\
2	0.0277202469886968\\
3	0.0114452422348174\\
4	0.00503045173826925\\
5	0.00127983615654185\\
6	0.000409844090094394\\
7	0.000556038061589787\\
8	0.000204925118437037\\
9	0.000441974099346351\\
10	2.21476558009799e-05\\
11	3.91890083913263e-05\\
12	1.91739693691686e-05\\
13	3.19820853360168e-06\\
14	8.39332878688404e-06\\
15	6.34970680127553e-06\\
16	3.75212853278728e-06\\
17	5.0670992636783e-06\\
18	1.63223697579291e-06\\
19	3.96455309593934e-07\\
20	3.82551833770957e-07\\
21	1.99583670260246e-07\\
22	4.03414111963312e-08\\
23	4.74577722412321e-08\\
24	5.86702550128733e-08\\
25	3.15947624099375e-08\\
26	5.26188942427961e-09\\
27	8.23268410077172e-09\\
28	8.65578153253254e-09\\
29	5.35932407165479e-09\\
30	4.19035698059003e-09\\
31	3.98613585295972e-09\\
32	3.80088955202929e-09\\
33	3.63215969977453e-09\\
34	3.47656495576565e-09\\
35	3.33165174042846e-09\\
36	3.59715951564158e-09\\
37	3.06899414898805e-09\\
38	2.95054855470801e-09\\
39	2.83927355713649e-09\\
40	2.73497267248653e-09\\
41	2.63696065173945e-09\\
42	2.5446242564241e-09\\
43	2.45757917763824e-09\\
44	2.37542267150147e-09\\
45	2.29775585255651e-09\\
46	2.22377032737588e-09\\
47	2.15430630754125e-09\\
48	2.11723683330475e-09\\
49	2.02515002907451e-09\\
};
\addlegendentry{Accelerated ADMM}

\end{axis}
\end{tikzpicture}%
\caption{Norm of primal residual versus number of iterations for the decentralized synthesis problem with standard and accelerated ADMM.}
\label{primal}
\end{figure}

%
%
\definecolor{mycolor1}{rgb}{0.00000,0.44700,0.74100}%
\definecolor{mycolor2}{rgb}{0.85000,0.32500,0.09800}%
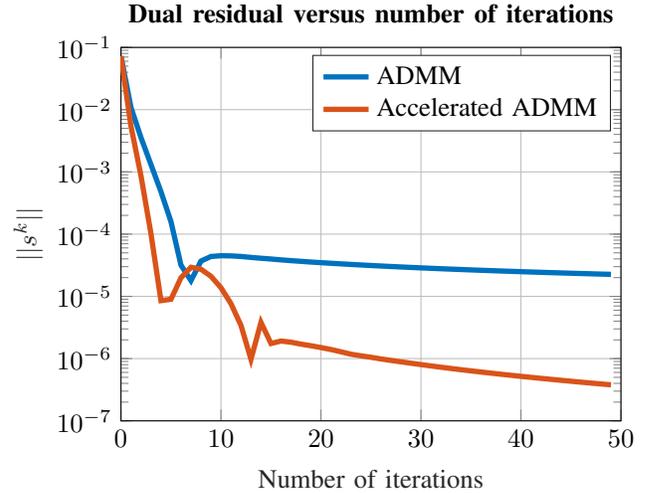
\begin{figure}[t!]
\begin{tikzpicture}

\begin{axis}[%
width=2.617in,
height=1.956in,
at={(0in,0in)},
scale only axis,
xmin=0,
xmax=50,
xlabel style={font=\color{white!15!black}},
xlabel={Number of iterations},
ymode=log,
ymin=1e-07,
ymax=0.1,
yminorticks=true,
xmajorgrids=true,
ymajorgrids=true,
ylabel style={font=\color{white!15!black}},
ylabel={$\vert\vert s^k \vert\vert$},
axis background/.style={fill=white},
title style={font=\bfseries},
title={Dual residual versus number of iterations},
legend style={legend cell align=left, align=left, draw=white!15!black}
]
\addplot [color=mycolor1, line width=2]
  table[row sep=crcr]{%
0	0.0720859424603907\\
1	0.0108513028658773\\
2	0.0035686009891074\\
3	0.00131124483249812\\
4	0.000480675047401737\\
5	0.000158726644419021\\
6	3.17425279922815e-05\\
7	1.79206013585656e-05\\
8	3.69250337064268e-05\\
9	4.37341493251442e-05\\
10	4.49710702964756e-05\\
11	4.46872771966619e-05\\
12	4.3769873137639e-05\\
13	4.2239178297272e-05\\
14	4.09609842512613e-05\\
15	3.98287560767496e-05\\
16	3.85717806299027e-05\\
17	3.75017866322694e-05\\
18	3.65572288056569e-05\\
19	3.55909078346223e-05\\
20	3.47799951196055e-05\\
21	3.40071276520688e-05\\
22	3.32791396257115e-05\\
23	3.25915264626013e-05\\
24	3.19415860315715e-05\\
25	3.13263029634374e-05\\
26	3.07432317457742e-05\\
27	3.01904830457413e-05\\
28	2.96654264275033e-05\\
29	2.91659927697705e-05\\
30	2.86906010435022e-05\\
31	2.82371737311717e-05\\
32	2.78048844891406e-05\\
33	2.73913641495637e-05\\
34	2.69957720112086e-05\\
35	2.66164851347561e-05\\
36	2.62527600125676e-05\\
37	2.59034649990494e-05\\
38	2.55675490515022e-05\\
39	2.52445410002423e-05\\
40	2.49333421760745e-05\\
41	2.46325062868826e-05\\
42	2.43433762058129e-05\\
43	2.40636396589892e-05\\
44	2.37934377948792e-05\\
45	2.3532121665096e-05\\
46	2.32789560265621e-05\\
47	2.30335336049257e-05\\
48	2.27953954026004e-05\\
49	2.25647536782804e-05\\
};
\addlegendentry{ADMM}

\addplot [color=mycolor2, line width=2]
  table[row sep=crcr]{%
0	0.0720859424603907\\
1	0.00542565143293864\\
2	0.000867925395597808\\
3	0.000101351405352896\\
4	8.53313362143551e-06\\
5	8.96913290603025e-06\\
6	1.98943887934433e-05\\
7	2.9327914877785e-05\\
8	2.72777159703234e-05\\
9	2.10141163504452e-05\\
10	1.37532455383172e-05\\
11	7.630339306816e-06\\
12	3.41644573093699e-06\\
13	9.76447259552614e-07\\
14	3.82300067603664e-06\\
15	1.74801476427455e-06\\
16	1.91498224782661e-06\\
17	1.83536244935804e-06\\
18	1.7074651437403e-06\\
19	1.60989502684613e-06\\
20	1.50437086216262e-06\\
21	1.39948420333064e-06\\
22	1.28669290444464e-06\\
23	1.17558829955311e-06\\
24	1.10826572247476e-06\\
25	1.05041681599324e-06\\
26	9.87681991017697e-07\\
27	9.34864786343096e-07\\
28	8.87230880813193e-07\\
29	8.42789028220101e-07\\
30	8.01344872815289e-07\\
31	7.62794007033406e-07\\
32	7.26986471633482e-07\\
33	6.93997578601836e-07\\
34	6.63429503192935e-07\\
35	6.35081091800047e-07\\
36	6.08691919538717e-07\\
37	5.83971270523919e-07\\
38	5.60849624605757e-07\\
39	5.3916218659357e-07\\
40	5.18729129473825e-07\\
41	4.99509519608389e-07\\
42	4.81388970783615e-07\\
43	4.64303725917092e-07\\
44	4.48127716783416e-07\\
45	4.32833662207804e-07\\
46	4.18282360353503e-07\\
47	4.04530404434984e-07\\
48	3.91494284233977e-07\\
49	3.79090268105843e-07\\
};
\addlegendentry{Accelerated ADMM}

\end{axis}
\end{tikzpicture}%
\caption{Norm of dual residual versus number of iterations for the decentralized synthesis problem with standard and accelerated ADMM.}
\label{dual}
\end{figure}

\section{CONCLUSIONS AND FUTURE WORK} \label{sec:conclusions}

We studied the distributed synthesis problem of large scale linear systems, for which an underlying interconnection topology is given. For such systems, we decompose the structured controller design problem (inherited by the interconnection topology) into a number of smaller local synthesis problems associated with local dissipation inequalities and a global gain condition. {Furthermore, we proposed a distributed optimization method with smoothing techniques, which enables to employ accelerated accelerated ADMM. Numerical results show that the accelerated ADMM method significantly improves the convergence rate compared to standard ADMM. } As discussed in Section~\ref{sec:conserv}, the block-diagonal Lyapunov function structure may lead to some degree of conservatism. Future research will explore other Lyapunov function structures that are less conservative and, at the same time, lead to a decompositional synthesis technique. In particular, the vector Lyapunov approach~\cite{7330553} seems promising in this respect. Future research will also consider the extensions of the distributed synthesis method to nonlinear (polynomial) systems and hybrid dynamical systems.

\bibliography{references}
\bibliographystyle{IEEEtran}


\appendix

\section{Uncertainties and Nonlinearities Described by Integral Quadratic Constraints}

In this section, we consider large-scale systems with subsystems that are subject to nonlinearities, time variations, and
uncertain parameters which  can be characterized by an integral quadratic constraint (IQC)~\cite{587335,yak71}. Let $\Gamma : j\mathbb{R} \to \mathbb{C}^{n_\mu \times n_\zeta}$ be a measurable Hermitian-valued function. We say $\Delta$ satisfies the IQC defined by $\Gamma $, if and only if the following inequality holds
\begin{equation} \label{dsdSXX}
\int_{-\infty}^{\infty} \begin{bmatrix} \hat{\mu}(j\omega) \\ \hat{\zeta}(j\omega) \end{bmatrix}^\prime \Gamma(j \omega) \begin{bmatrix} \hat{\mu}(j\omega) \\ \hat{\zeta}(j\omega) \end{bmatrix} ~d \omega \ge 0,
\end{equation}
where $\mu \in L_2$, $\zeta = \Delta(\mu)$, and $\hat{\mu}(j\omega)$ and $\hat{\zeta}(j\omega)$ are the Fourier transforms of $\mu$ and $\zeta$, respectively. If the IQC multiplier $\Gamma$ is uniformly bounded on the imaginary axis and rational, then \eqref{dsdSXX} can also be described in the time-domain, wherein $\Gamma$ can be factorized as $\Psi(j\omega)^* Q  \Psi(j\omega)$, with $Q$ being a constant matrix and $\Psi$ being a stable linear time invariant system. Let $\Psi$ have the realization
\begin{align}
\Psi : \begin{cases}
\dot{\gamma} = \tilde{A} \gamma + \tilde{B} \mu + \tilde{G}\zeta,\\
y_\psi = \tilde{C} \gamma  + \tilde{D} \mu + \tilde{H}\zeta\\
\gamma(0)=0.
\end{cases}
\end{align}
Since our formulation is based on dissipation inequalities, we consider the so called "hard" IQCs~\cite{6915700} defined by $(\Psi,Q)$. That is, IQCs of the form
\begin{equation}
\int_0^T y_\psi^\prime Q  y_\psi~dt \ge 0,
\end{equation}
for any $0\le T<\infty$, $ y_\psi$ is defined as
$$
y_\psi = \Psi \begin{bmatrix} \mu \\ \zeta \end{bmatrix}.
$$ 

To account for IQC-bounded uncertainties, we consider  local storage functions $V_i(x_i,\gamma_i) = \begin{bmatrix} x_i \\ \gamma_i \end{bmatrix}^\prime P_i \begin{bmatrix} x_i \\ \gamma_i \end{bmatrix}$ with $P_i$ of appropriate dimension depending on the realization of $\Psi_i$. Then, in order for the subsystems  to be dissipative subject to an IQC, condition~\eqref{ddsdsdsdsdsccdcdxx} should be modified to 
\begin{multline}
\left(\frac{\partial V_i}{\partial x_i}\right)^\prime \left(A_i x_i + B_i u_i + G_i w_i \right) \\+ \left(\frac{\partial V_i}{\partial \gamma_i}\right)^\prime \left(\tilde{A}_i \gamma_i + \tilde{B}_i \mu_i + \tilde{G}_i\zeta_i\right) \le 
\begin{bmatrix} w_i \\ y_i\end{bmatrix}^\prime 
S_i
\begin{bmatrix} w_i \\ y_i \end{bmatrix} \\
-\lambda_i \left(\tilde{C}_i \gamma_i  + \tilde{D}_i \mu_i + \tilde{H}_i\zeta_i\right)^\prime Q_i \left(\tilde{C}_i \gamma_i  + \tilde{D}_i \mu_i + \tilde{H}_i\zeta_i\right)
\end{multline}
for some $\lambda_i>0$. 

Then, Lemma~1 in~\cite{6915700}  can be used to formulate LMI conditions for checking dissipativity of the subsystem subject to an uncertainty or nonlinearity defined by an IQC.

\end{document}